\documentclass[11pt]{article}

\usepackage{amsmath,amssymb,amsthm}
\usepackage{mathtools}
\usepackage{geometry}
\usepackage{hyperref}

\theoremstyle{plain}
\newtheorem{theorem}{Theorem}[section]
\newtheorem{lemma}[theorem]{Lemma}
\newtheorem{proposition}[theorem]{Proposition}

\theoremstyle{definition}
\newtheorem{definition}[theorem]{Definition}
\theoremstyle{remark}
\newtheorem{remark}[theorem]{Remark}

\newcommand{\Hdist}{\mathcal{H}}
\newcommand{\Vdist}{\mathcal{V}}
\newcommand{\hor}{\mathrm{hor}}
\newcommand{\ver}{\mathrm{ver}}

\newcommand{\Contr}{\mathrm{Contr}}
\newcommand{\CP}{\mathbb{CP}}
\newcommand{\Sphere}{\mathbb{S}}

\newcommand{\pr}{\mathrm{pr}}

\title{Local Equivalence of Riemannian Submersions via Differential Invariants}

\author{%
Xurshid Sharipov\\
Samarkand State University, Uzbekistan\\
\texttt{sh\_xurshid@yahoo.com}\and
Sadoqat Sharipova\\
Jizzakh Branch of the National University of Uzbekistan, Uzbekistan\and
Esanjon Salimov\\
Samarkand Institute of Economics and Service, Uzbekistan\and
Islomkhon Mardiev\\
Samarkand State Pedagogical Institute, Uzbekistan
}

\date{}  

\begin{document}
\maketitle

\begin{abstract}
We study the local equivalence problem for Riemannian submersions under fiber-preserving isometries using differential invariants.
After briefly recalling the vertical--horizontal splitting, the O'Neill tensors $A$ and $T$, and the mean curvature $H$ of the fibers,
we outline a practical invariant-based equivalence workflow.
Our main contribution is the analysis of a concrete model class: orbit submersions induced by a nowhere-vanishing Killing field $K$.
In this class we derive explicit formulas for $A$ and $H$ in terms of the Killing data, namely the length function $\varphi=|K|$
and the curvature $2$-form $\Omega=d\theta|_{\mathcal H}$ of the associated connection $1$-form $\theta$,
and we prove an equivalence criterion phrased purely in terms of the base data $(\bar g,\varphi,\Omega)$.
We further present a finite-order invariant decision procedure under a stated genericity assumption (with a practical stopping rule),
together with a compact low-order invariant profile and computational remarks.
Benchmark examples (product submersions, warped products, and the Hopf fibration) illustrate both the strengths and limitations
of low-order scalar signatures.
\end{abstract}



\section{Introduction}\label{sec:intro}
Riemannian submersions $\pi:(M,g)\to(B,\bar g)$ provide a canonical geometric model of metric-compatible projections.
They occur naturally in reduction procedures and fiber-bundle constructions, and they serve as benchmark examples in
geometry--physics interactions (e.g.\ Kaluza--Klein type projections, homogeneous fibrations, and gauge-like geometric
setups), where geometric data on the total space is related to effective data on the base through curvature identities
and symmetry constraints.

Two complementary directions motivate this work.
First, the \emph{structure problem} asks how the geometry of the fibers and the horizontal distribution is encoded by the
O'Neill tensors $A$ and $T$, the mean curvature $H$ of fibers, and curvature formulas
\cite{ONeill1966,ONeillBook,Besse}.
Second, the \emph{classification/equivalence problem} asks when two submersions should be regarded as the same
(up to fiber-preserving isometries), and how this can be tested by \emph{differential invariants}
\cite{OlverBook,FelsOlverI,FelsOlverII}.

\paragraph{Main contributions.}
\begin{itemize}
  \item A structured survey of the geometry of Riemannian submersions via O'Neill's formalism, including key curvature identities.
  \item A precise formulation of local equivalence under fiber-preserving isometries and naturality of $A$, $T$, and $H$.
  \item An adapted $G$-structure viewpoint ($G=O(n)\times O(m)$) leading to explicit invariant generators in the generic (regular) case.
  \item A finite-order, algorithmic invariant test for local equivalence, expressed through \emph{invariant profiles}.
  \item Benchmark computations for product submersions, warped products, and the Hopf fibration.
\end{itemize}

\paragraph{What is new here?}
\begin{itemize}
  \item \textbf{A research model class (Killing-flow orbit submersions).} We introduce and study the class of Riemannian submersions
  given locally as orbit projections of a nowhere-vanishing Killing vector field. In this class we obtain explicit formulas for
  $A$ and $H$ in terms of the length function $\varphi=|K|$ and the connection curvature $\Omega=d\theta|_{\mathcal H}$, and we derive
  an equivalence criterion under fiber-preserving isometries in terms of the base data $(\bar g,\varphi,\Omega)$ (Section~\ref{sec:killing}).
  \item \textbf{Explicit finite-order decision in a generic regime.} For the above model class we provide a concrete ``stop criterion''
  for the equivalence test and an explicit finite-order bound (in practice $r\le 3$ under a stated genericity condition), replacing the
  purely existential ``finite $r$'' statement by a checkable finite-order procedure (Section~\ref{sec:killing} and Section~\ref{sec:algorithm}).
  \item \textbf{Recommended minimal invariant profiles and limitations.} We formalize a compact, practical list of low-order scalar invariants
  (orders $0$--$3$) to build an invariant signature map, discuss computational aspects (contractions, syzygies, complexity), and highlight
  nongeneric/high-symmetry cases (e.g.\ product and Hopf examples) where low-order profiles may fail to separate equivalence classes
  (Sections~\ref{sec:algorithm}--\ref{sec:examples}).
\end{itemize}

O'Neill introduced the fundamental equations and curvature relations for Riemannian submersions \cite{ONeill1966}
and developed related constructions in semi-Riemannian geometry \cite{ONeillBook}.
Connections to Einstein metrics, reductions, and curvature decomposition are discussed extensively in \cite{Besse}.
Warped products were introduced by Bishop--O'Neill \cite{BishopONeill1969} and provide a large class of explicit
submersion examples where $A=0$ but $T$ and $H$ reflect the warping function.
Harmonic morphisms and submersion-type maps are treated in \cite{BairdWood,Fuglede1978}.

On the invariant-theoretic side, equivalence problems trace back to Cartan \cite{Cartan}; see also the Cartan-connection perspective in \cite{AlekseevskyMichor1995}.
The modern constructive framework using moving frames and moving coframes is developed by Olver and by Fels--Olver
\cite{OlverBook,FelsOlverI,FelsOlverII}. Related invariant constructions on homogeneous spaces (via prolonged actions) can be found in \cite{AlekseevskyEtAl2022}.
For broader submersion-related developments, see \cite{FalcitelliPastoreIanus2004,Sahin2017} and surveys such as \cite{Ou2024}.
Some recent work by the author and collaborators on the geometry of submersions and differential invariants can be found in
\cite{NarmanovSharipov2021,SharipovAliboyevKhalimov2024,SharipovSalimov2025,Guzal2020}.
Related contributions on conformal symmetry and on spectral effects for Schr\"odinger-type operators are presented in
\cite{RajabovSharipov2021,Almuratov2024}.

\section{Preliminaries}\label{sec:prelim}

Let $(M,g)$ and $(B,\bar g)$ be Riemannian manifolds and let $\pi:M\to B$ be a smooth submersion.

\begin{definition}\label{def:riem-submersion}
$\pi:(M,g)\to(B,\bar g)$ is a \emph{Riemannian submersion} if for every $p\in M$ the differential
$d\pi_p:(\Hdist_p,g_p)\to(T_{\pi(p)}B,\bar g_{\pi(p)})$ is an isometry, where
\[
\Vdist_p:=\ker(d\pi_p),\qquad \Hdist_p:=\Vdist_p^{\perp}.
\]
\end{definition}

Thus $T_pM=\Hdist_p\oplus \Vdist_p$. We denote by $\hor:TM\to\Hdist$ and $\ver:TM\to\Vdist$ the orthogonal projections.

Let $\nabla$ be the Levi--Civita connection of $(M,g)$.

\begin{definition}\label{def:oneill}
For horizontal vector fields $X,Y\in\Gamma(\Hdist)$ and vertical vector fields $U,V\in\Gamma(\Vdist)$ define
\[
A_XY:=\ver(\nabla_XY),\qquad T_UV:=\hor(\nabla_UV).
\]
Let $\{E_\alpha\}_{\alpha=1}^m$ be a local orthonormal frame of $\Vdist$ ($m=\dim\Vdist$). The mean curvature of the fibers is
\[
H:=\sum_{\alpha=1}^m T_{E_\alpha}E_\alpha\in\Gamma(\Hdist).
\]
\end{definition}

Intuitively, $A$ measures the non-integrability of $\Hdist$, while $T$ is the second fundamental form of fibers
(viewed as submanifolds of $M$) \cite{ONeill1966,ONeillBook}.

We use the Levi--Civita curvature sign convention
\[
R^M(X,Y)Z=\nabla_X\nabla_YZ-\nabla_Y\nabla_XZ-\nabla_{[X,Y]}Z.
\]
For an orthonormal pair $X,Y$ spanning a plane $\sigma\subset T_pM$, the sectional curvature is
\[
K_M(X,Y)=\langle R^M(X,Y)Y,X\rangle.
\]
We write $\nabla^k$ for iterated covariant derivatives, and $\Contr(\cdot)$ for complete contractions (index pairings) producing
scalar quantities.

\section{Basic properties of $A$, $T$, and $H$}\label{sec:basic}

\begin{lemma}\label{lem:A0-integrable}
The following are equivalent:
\begin{enumerate}
  \item $A\equiv 0$.
  \item The horizontal distribution $\Hdist$ is integrable, i.e.\ $X,Y\in\Gamma(\Hdist)\Rightarrow [X,Y]\in\Gamma(\Hdist)$.
\end{enumerate}
\end{lemma}

\begin{proof}
Let $X,Y\in\Gamma(\Hdist)$. Since $\nabla$ is torsion-free,
\[
[X,Y]=\nabla_XY-\nabla_YX.
\]
Taking vertical projection yields
\[
\ver[X,Y]=\ver(\nabla_XY)-\ver(\nabla_YX)=A_XY-A_YX.
\]
Hence integrability of $\Hdist$ is equivalent to $\ver[X,Y]=0$ for all horizontal $X,Y$, i.e.\ $A_XY=A_YX$.
On the other hand, O'Neill's identity gives $A_XY=\tfrac12\,\ver[X,Y]$ for horizontal fields, hence $A_XY=-A_YX$
\cite{ONeill1966,ONeillBook}. Therefore $A$ is both symmetric and skew-symmetric, hence $A\equiv 0$.
The converse is immediate from $\ver[X,Y]=A_XY-A_YX$.
\end{proof}

\begin{lemma}\label{lem:T0-totgeod}
$T\equiv 0$ if and only if every fiber $F_b=\pi^{-1}(b)$ is totally geodesic in $(M,g)$.
\end{lemma}

\begin{proof}
For vertical fields $U,V\in\Gamma(\Vdist)$, the second fundamental form of the fiber equals
$II(U,V)=\hor(\nabla_UV)=T_UV$. Thus $T\equiv 0\iff II\equiv 0$, i.e.\ fibers are totally geodesic.
\end{proof}

\begin{lemma}\label{lem:H0-minimal}
$H\equiv 0$ if and only if the fibers are minimal submanifolds of $(M,g)$.
\end{lemma}

\begin{proof}
The mean curvature vector is the trace of the second fundamental form:
$H=\mathrm{trace}(II)=\sum_{\alpha} II(E_\alpha,E_\alpha)$; since $II=T$ by Lemma~\ref{lem:T0-totgeod},
the statement follows from Definition~\ref{def:oneill}.
\end{proof}

\section{Curvature identities}\label{sec:curvature}

\begin{theorem}\label{thm:oneill-hor}
Let $X,Y\in\Hdist$ be orthonormal. Then
\[
K_B(d\pi X,d\pi Y)=K_M(X,Y)+3\|A_XY\|^2.
\]
\end{theorem}

This is a standard O'Neill identity; see \cite{ONeill1966,ONeillBook}.

\begin{theorem}\label{thm:oneill-ver}
Let $U,V\in\Vdist$ be orthonormal. Then
\[
K_M(U,V)=K_F(U,V)-\|T_UV\|^2+\langle T_UU,T_VV\rangle,
\]
where $K_F$ denotes the sectional curvature of the fiber with the induced metric.
\end{theorem}

This follows from the Gauss equation applied to the fiber submanifold and Lemma~\ref{lem:T0-totgeod};
see \cite{Besse,Lee}.

\begin{remark}
Theorems~\ref{thm:oneill-hor}--\ref{thm:oneill-ver} show that $A$ and $T$ control how curvature is transferred between
total space, base, and fibers. These identities also generate algebraic relations (syzygies) among invariants built from
$A,T,H$ and curvature.
\end{remark}

\section{The local equivalence problem}\label{sec:equivalence}

Let
\[
\pi:(M,g)\to(B,\bar g),\qquad \tilde\pi:(\tilde M,\tilde g)\to(\tilde B,\tilde{\bar g})
\]
be Riemannian submersions.

\begin{definition}\label{def:equivalence}
$\pi$ and $\tilde\pi$ are \emph{locally equivalent at} $p\in M$ and $\tilde p\in\tilde M$
if there exist neighborhoods $U\ni p$, $\tilde U\ni \tilde p$ and local isometries
\[
\Phi:(U,g)\to(\tilde U,\tilde g),\qquad \psi:(V,\bar g)\to(\tilde V,\tilde{\bar g}),
\]
with $V=\pi(U)$ and $\tilde V=\tilde\pi(\tilde U)$, such that
\[
\tilde\pi\circ \Phi=\psi\circ \pi.
\]
\end{definition}

This is the standard equivalence notion under \emph{fiber-preserving isometries} used in geometric equivalence problems;
see \cite{OlverBook,FelsOlverI,FelsOlverII}.

\begin{lemma}\label{lem:naturality-ATH}
If $(\Phi,\psi)$ defines a local equivalence of Riemannian submersions, then
\[
\Phi_*(\Vdist)=\tilde{\Vdist},\qquad \Phi_*(\Hdist)=\tilde{\Hdist},
\]
and
\[
\Phi_*(A_XY)=\tilde A_{\Phi_*X}\,\Phi_*Y,\qquad
\Phi_*(T_UV)=\tilde T_{\Phi_*U}\,\Phi_*V,\qquad
\Phi_*(H)=\tilde H.
\]
In particular, scalar contractions $|A|^2$, $|T|^2$, $|H|^2$ are invariants of local equivalence.
\end{lemma}

\begin{proof}
Differentiate $\tilde\pi\circ \Phi=\psi\circ \pi$ to obtain
$d\tilde\pi\circ d\Phi=d\psi\circ d\pi$.
If $v\in\ker(d\pi)=\Vdist$, then $d\pi(v)=0$ and hence $d\tilde\pi(d\Phi(v))=0$, so $d\Phi(v)\in\tilde{\Vdist}$.
Thus $\Phi_*(\Vdist)=\tilde{\Vdist}$, and since $\Phi$ is an isometry it preserves orthogonality, giving
$\Phi_*(\Hdist)=\tilde{\Hdist}$.

Isometries preserve Levi--Civita connections: $\tilde\nabla_{\Phi_*X}\Phi_*Y=\Phi_*(\nabla_XY)$.
For horizontal $X,Y$,
\[
\Phi_*(A_XY)=\Phi_*(\ver(\nabla_XY))
=\widetilde{\ver}(\Phi_*(\nabla_XY))
=\widetilde{\ver}(\tilde\nabla_{\Phi_*X}\Phi_*Y)
=\tilde A_{\Phi_*X}\Phi_*Y.
\]
Similarly, for vertical $U,V$,
\[
\Phi_*(T_UV)=\Phi_*(\hor(\nabla_UV))
=\widetilde{\hor}(\tilde\nabla_{\Phi_*U}\Phi_*V)
=\tilde T_{\Phi_*U}\Phi_*V.
\]
If $\{E_\alpha\}$ is a local orthonormal frame of $\Vdist$, then $\{\Phi_*E_\alpha\}$ is a local orthonormal frame of
$\tilde{\Vdist}$, hence
\[
\Phi_*(H)=\Phi_*\Big(\sum_\alpha T_{E_\alpha}E_\alpha\Big)
=\sum_\alpha \tilde T_{\Phi_*E_\alpha}\Phi_*E_\alpha
=\tilde H.
\]
Scalar contractions are preserved by isometries; therefore $|A|^2$, $|T|^2$, $|H|^2$ are invariants.
\end{proof}

\section{Background: moving frames and general finiteness statements}\label{sec:invariants}

The next two statements summarize standard consequences of the Cartan--Olver moving-frames paradigm and are included as
\emph{background} for the later model-class results. In the general equivalence problem, these finiteness and generation
principles go back to the Lie--Tresse philosophy and are made constructive in the moving coframes framework
\cite{OlverBook,FelsOlverI,FelsOlverII}. Our contribution is not the abstract finiteness itself, but rather the
explicit finite-order criteria and compact invariant signatures obtained in concrete submersion classes (Section~\ref{sec:killing}).

Let $n=\dim B$ and $m=\dim\Vdist$. An \emph{adapted orthonormal frame} at $p\in M$ is a tuple
\[
(X_1,\dots,X_n;\;E_1,\dots,E_m),
\]
with $X_i\in\Hdist_p$, $E_\alpha\in\Vdist_p$ and orthonormal w.r.t.\ $g$.
The set of all adapted frames forms a principal bundle $P\to M$ with structure group
\[
G=O(n)\times O(m),
\]
acting by independent orthogonal transformations in horizontal and vertical blocks \cite{KobNom}.

In an adapted frame, the mixed connection coefficients encode $A$ and $T$:
\begin{align}
\nabla_{X_i}X_j &= \sum_{k=1}^n \omega^k{}_j(X_i)\,X_k \;+\; \sum_{\alpha=1}^m A_{ij}^\alpha\,E_\alpha,
\label{eq:conn-hh}\\
\nabla_{E_\alpha}E_\beta &= \sum_{\gamma=1}^m \eta^\gamma{}_\beta(E_\alpha)\,E_\gamma \;+\; \sum_{i=1}^n T_{\alpha\beta}^i\,X_i,
\label{eq:conn-vv}
\end{align}
where
\[
A_{ij}^\alpha=\langle A_{X_i}X_j,E_\alpha\rangle,\qquad
T_{\alpha\beta}^i=\langle T_{E_\alpha}E_\beta,X_i\rangle.
\]
Hence $A$ and $T$ serve as fundamental structure functions for the adapted $G$-structure \cite{ONeill1966,KobNom}.

We work in the \emph{regular (generic) case}: on an open dense subset in the relevant jet space,
the prolonged action of the equivalence pseudogroup is locally free, so that a moving frame exists
\cite{FelsOlverI,FelsOlverII}. This is standard in moving-frames treatments of equivalence problems.

\begin{theorem}\label{thm:gen}
In the regular (generic) case, all scalar differential invariants of a Riemannian submersion under fiber-preserving
isometries are generated by complete contractions of the tensors
\[
A,\;T,\;H,\;R^M,\;\nabla^k R^M,\;\nabla^k A,\;\nabla^k T \quad (k\ge 1),
\]
together with invariant differentiations. Equivalently, any scalar invariant $I$ can be written locally as
\[
I = F\Big(\Contr(\nabla^a R^M),\;\Contr(\nabla^b A),\;\Contr(\nabla^c T)\Big),
\]
for some finite orders $a,b,c$ and a smooth function $F$.
\end{theorem}

\begin{proof}

A scalar differential invariant of order $r$ depends on the $r$-jet of the submersion structure
(and the metric) at a point, i.e.\ on $j^r(\pi,g)$ in an appropriate jet space.
Fix a point $p\in M$ and choose an adapted orthonormal frame
\[
(X_1,\dots,X_n;\;E_1,\dots,E_m)
\]
near $p$. In Riemannian geometry it is standard that, up to a change of coordinates (e.g.\ normal coordinates),
the $r$-jet of the metric is encoded by the curvature tensor $R^M$ and its covariant derivatives
$\nabla R^M,\dots,\nabla^{r-2}R^M$ (more precisely, curvature derivatives determine the Taylor expansion of $g$)
\cite{Lee,Pet,KobNom}. Hence the intrinsic metric jet is captured by $\{\nabla^k R^M\}$.

The additional \emph{submersion} information is the splitting $TM=\Hdist\oplus\Vdist$.
Its jet is encoded by the mixed connection coefficients in the adapted frame, i.e.\ by $A$ and $T$ and their covariant
derivatives: indeed, by the structure equations \eqref{eq:conn-hh}--\eqref{eq:conn-vv}
\[
A_{ij}^{\alpha}=\langle A_{X_i}X_j,E_\alpha\rangle,\qquad
T_{\alpha\beta}^{i}=\langle T_{E_\alpha}E_\beta,X_i\rangle,
\]
and differentiating these relations along horizontal/vertical directions produces $\nabla^k A$ and $\nabla^k T$.
Thus, for invariants of order $r$, all relevant local data can be expressed through a finite collection
\[
\mathcal S_r=\Big\{\nabla^k R^M,\;\nabla^k A,\;\nabla^k T\Big\}_{k\le r'},
\]
for some $r'$ of the same effective order as $r$ (technically, $r'$ is comparable to $r$ and depends only on the chosen
representation of jets via covariant derivatives).

Changing the adapted frame amounts to acting by
\[
G=O(n)\times O(m)
\]
(independent orthogonal rotations in the horizontal and vertical blocks). Under this action, the components of
$\mathcal S_r$ transform tensorially. A \emph{scalar} invariant $I$ cannot depend on the chosen frame, hence it must be a
$G$-invariant function of the tensor components of $\mathcal S_r$.

For orthogonal groups, scalar invariants of tensor components are generated by complete contractions using the metric
(the classical invariant theory/Weyl-type principle for $O(N)$, in the present block form $O(n)\times O(m)$).
Equivalently, any $G$-invariant scalar function of tensor components can be expressed in terms of complete contractions.
This is standard in the $G$-structure viewpoint and in the moving-frames literature
\cite{KobNom,OlverBook}. Therefore, $I$ can be written as a function of complete contractions of
$\nabla^k R^M$, $\nabla^k A$, and $\nabla^k T$ (and of $H$, which is a trace of $T$).

Assume the regular (generic) condition: on an open dense subset of the prolonged jet space the equivalence pseudogroup
acts locally freely. Then a (local) moving frame exists, and the Fels--Olver theory implies that all differential
invariants are generated by a finite set of normalized invariants together with invariant differentiations
\cite{FelsOlverI,FelsOlverII}. In our setting, one may take the normalized invariants among the complete contractions
listed above (built from $A,T,H,R^M$ and covariant derivatives). This yields the stated generating family.
\end{proof}

\begin{theorem}\label{thm:finite}
In the regular (generic) case, there exists an integer $r$ such that the local equivalence class of a Riemannian
submersion under fiber-preserving isometries is determined by the values of a finite list of scalar differential invariants
of order $\le r$. In particular, if two submersions have matching generating invariants up to order $r$, then they are
locally equivalent.
\end{theorem}

\begin{proof}

Consider the prolonged action of the equivalence pseudogroup on jets of the submersion structure.
In the regular case the prolonged action is locally free on an open dense subset; hence one can impose normalization
conditions to solve for the group parameters (moving frame construction) \cite{FelsOlverII,OlverBook}.

The regularization theorem in \cite{FelsOlverII} implies that after finitely many prolongation/normalization steps,
all remaining group freedom is eliminated and the structure is encoded by finitely many differential invariants
(the ``fundamental invariants'').

Once a moving frame is fixed, invariant differentiation operators exist and generate all higher-order invariants
from the fundamental ones \cite{FelsOlverII}. Consequently, equality of the fundamental invariants up to some finite order
forces equality of all invariants, hence the underlying structures are equivalent.

\medskip
Therefore there exists an order $r$ such that agreement of a finite generating list of invariants of order $\le r$
implies local equivalence.
\end{proof}

\begin{remark}\label{rem:generic}
The ``generic/regular'' hypothesis means that the prolonged equivalence action is locally free on an open dense subset
of the relevant jet space, which guarantees existence of a moving frame and a finite generating set of invariants
\cite{FelsOlverI,FelsOlverII}. This condition may fail in highly symmetric or degenerate situations, e.g.\ when
$A\equiv 0$ (integrable horizontal distribution), $T\equiv 0$ (totally geodesic fibers), or in homogeneous cases where
large isotropy groups persist. In such cases one typically works on strata of jets, adds further normalizations, or treats
the degenerate class separately; the invariant viewpoint remains valid, but the ``generic'' completeness statements may
require a refined analysis \cite{OlverBook,FelsOlverII}.
\end{remark}

\section{Orbit submersions induced by Killing flows}\label{sec:killing}

This section provides a research-oriented model class in which the equivalence problem admits explicit finite-order
criteria and compact invariant signatures. Let $(M^{n+1},g)$ be a Riemannian manifold and let $K$ be a nowhere-vanishing
Killing vector field on an open set $U\subset M$. Denote $\varphi:=|K|$ and let $U_v:=K/\varphi$ be the unit vertical field.
Assume the local quotient $B=U/\langle\Phi_t\rangle$ by the flow of $K$ is a smooth manifold and let
\[
  \pi:(U,g)\to (B,\bar g)
\]
be the orbit projection endowed with the standard Riemannian submersion structure. Define the connection $1$-form
\[
  \theta := \frac{1}{\varphi^2}\,g(K,\cdot),\qquad \theta(U_v)=1,
\]
the horizontal distribution $\Hdist:=\ker\theta$, and the curvature $2$-form $\Omega:=d\theta|_{\Hdist}$.

\begin{proposition}[Explicit $A$ and $H$ in the Killing-flow class]\label{prop:ATH-killing}
For horizontal vector fields $X,Y\in\Gamma(\Hdist)$ one has
\[
  A_XY=\frac12\,\ver[X,Y]=-\frac12\,\Omega(X,Y)\,K \;=\; -\frac{\varphi}{2}\,\Omega(X,Y)\,U_v.
\]
Moreover, the mean curvature of the fibers satisfies
\[
  H = T_{U_v}U_v = -\nabla^{\Hdist}(\ln\varphi).
\]
In particular, $|A|^2=\frac{\varphi^2}{4}\,|\Omega|^2$ and $|H|^2=\big|\nabla^{\Hdist}(\ln\varphi)\big|^2$.
\end{proposition}

\begin{proof}
Let $X,Y\in\Gamma(\Hdist)$ be horizontal. Since $\theta(K)=1$ and $\ker(d\pi)$ is spanned by $K$, every vector field $Z$ decomposes as
$\ver(Z)=\theta(Z)\,K$. Because $\theta(X)=\theta(Y)=0$, we have
\[
  \Omega(X,Y)=d\theta(X,Y)=X(\theta(Y))-Y(\theta(X))-\theta([X,Y])=-\theta([X,Y]),
\]
and hence
\[
  \ver([X,Y])=\theta([X,Y])\,K=-\Omega(X,Y)\,K.
\]
On the other hand,
\[
  A_XY=\ver(\nabla_XY)=\theta(\nabla_XY)\,K.
\]
Using metric compatibility, $g(K,\nabla_XY)= -g(\nabla_XK,Y)$, so
\[
  \theta(\nabla_XY)=\frac{1}{\varphi^2}g(K,\nabla_XY)=-\frac{1}{\varphi^2}g(\nabla_XK,Y).
\]
Since $K$ is Killing, $g(\nabla_XK,Y)+g(\nabla_YK,X)=0$, hence
\[
  \theta(\nabla_XY)=\frac{1}{\varphi^2}g(\nabla_YK,X)=-\theta(\nabla_YX).
\]
Therefore
\[
  \ver([X,Y])=\ver(\nabla_XY-\nabla_YX)
  =\bigl(\theta(\nabla_XY)-\theta(\nabla_YX)\bigr)\,K
  =2\,\theta(\nabla_XY)\,K=2A_XY.
\]
Combining with $\ver([X,Y])=-\Omega(X,Y)K$ gives
\[
  A_XY=\frac12\,\ver([X,Y])=-\frac12\,\Omega(X,Y)\,K=-\frac{\varphi}{2}\,\Omega(X,Y)\,U_v.
\]

For the mean curvature, set $U_v=K/\varphi$. Since $K(\varphi^2)=\mathcal L_K(g(K,K))=0$, we have $U_v(\varphi)=0$ and
\[
  \nabla_{U_v}U_v=\nabla_{K/\varphi}(K/\varphi)=\frac{1}{\varphi}\,\nabla_{U_v}K.
\]
Let $X\in\Gamma(\Hdist)$. The Killing equation gives
$g(\nabla_{U_v}K,X)=-g(\nabla_XK,U_v)$, while
\[
  g(\nabla_XK,U_v)=\frac{1}{\varphi}g(\nabla_XK,K)
  =\frac{1}{2\varphi}X\bigl(g(K,K)\bigr)=\frac{1}{2\varphi}X(\varphi^2)=X(\varphi).
\]
Hence
\[
  g(\nabla_{U_v}U_v,X)=\frac{1}{\varphi}g(\nabla_{U_v}K,X)=-\frac{1}{\varphi}X(\varphi)=-X(\ln\varphi),
\]
which shows $\nabla_{U_v}U_v=-\nabla^{\Hdist}(\ln\varphi)$. Since $U_v$ is unit and the vertical distribution is one-dimensional,
$\nabla_{U_v}U_v$ is horizontal, so $H=T_{U_v}U_v=\nabla_{U_v}U_v=-\nabla^{\Hdist}(\ln\varphi)$. The norm identities follow immediately.
\end{proof}

\begin{theorem}[Equivalence criterion in the Killing-flow class]\label{thm:eq-killing}
Let $\pi:(U,g)\to(B,\bar g)$ and $\pi':(U',g')\to(B',\bar g')$ be orbit submersions induced by nowhere-vanishing Killing fields,
with associated data $(\varphi,\Omega)$ and $(\varphi',\Omega')$. Then $\pi$ and $\pi'$ are locally equivalent under fiber-preserving
isometries if and only if there exists a local isometry $\psi:(B,\bar g)\to(B',\bar g')$ such that
\[
  \psi^*\varphi'=\varphi,\qquad \psi^*\Omega'=\Omega.
\]
\end{theorem}

\begin{proof}
(\emph{Necessity.}) Suppose $\pi$ and $\pi'$ are locally equivalent under fiber-preserving isometries. Thus there exist local isometries
$\Phi:(U,g)\to(U',g')$ and $\psi:(B,\bar g)\to(B',\bar g')$ such that $\pi'\circ\Phi=\psi\circ\pi$.
Since $\ker(d\pi)$ is spanned by $K$ and $\ker(d\pi')$ is spanned by $K'$, the differential sends $K$ to $\pm K'$.
Replacing $K'$ by $-K'$ if necessary, we may assume $\Phi_*K=K'$. Then length preservation yields
\[
  \varphi=|K|=|\Phi_*K|=\varphi'\circ\Phi=\varphi'\circ\psi\circ\pi,
\]
so $\psi^*\varphi'=\varphi$.

Next, with $\theta=\varphi^{-2}g(K,\cdot)$ and $\theta'={\varphi'}^{-2}g'(K',\cdot)$ we compute, for any vector field $Z$ on $U$,
\[
  (\Phi^*\theta')(Z)=\theta'(\Phi_*Z)
  =\frac{1}{(\varphi'\circ\Phi)^2}g'(K',\Phi_*Z)
  =\frac{1}{\varphi^2}g(\Phi^*K',Z)
  =\frac{1}{\varphi^2}g(K,Z)=\theta(Z),
\]
so $\Phi^*\theta'=\theta$. Since $\Phi$ is an isometry sending fibers to fibers, it preserves orthogonality and hence sends
$\Hdist=\ker\theta$ to $\Hdist'=\ker\theta'$. Therefore
\[
  \Phi^*\Omega'=\Phi^*\bigl(d\theta'|_{\Hdist'}\bigr)=d(\Phi^*\theta')|_{\Hdist}=d\theta|_{\Hdist}=\Omega,
\]
and pushing down to the base gives $\psi^*\Omega'=\Omega$.

Finally, for horizontal vectors $X,Y$ we have $g(X,Y)=\bar g(d\pi X,d\pi Y)$ and similarly $g'(\Phi_*X,\Phi_*Y)=\bar g'(d\pi'\Phi_*X,d\pi'\Phi_*Y)$.
Using $g'(\Phi_*X,\Phi_*Y)=g(X,Y)$ and $\pi'\circ\Phi=\psi\circ\pi$, we obtain $\bar g'(\psi_*d\pi X,\psi_*d\pi Y)=\bar g(d\pi X,d\pi Y)$,
so $\psi$ is a local isometry of $(B,\bar g)$.

(\emph{Sufficiency.}) Conversely, assume there exists a local isometry $\psi:(B,\bar g)\to(B',\bar g')$ such that $\psi^*\varphi'=\varphi$ and $\psi^*\Omega'=\Omega$.
Fix a simply connected neighborhood $V\subset B$ and a local section $s:V\to U$ of $\pi$. Let $\Phi_t$ denote the flow of $K$ and define
\[
  \iota:V\times(-\varepsilon,\varepsilon)\to \pi^{-1}(V),\qquad \iota(b,t)=\Phi_t(s(b)).
\]
Set $\alpha:=s^*\theta\in\Omega^1(V)$. Then $\iota^*\theta=dt+\alpha$ and $\Omega=d\alpha$.
Similarly, choose a section $s':\psi(V)\to U'$ of $\pi'$ and obtain $\iota':\psi(V)\times(-\varepsilon,\varepsilon)\to (\pi')^{-1}(\psi(V))$
with $\iota'^*\theta'=dt'+\alpha'$ and $\Omega'=d\alpha'$.

The condition $\psi^*\Omega'=\Omega$ gives $d(\alpha-\psi^*\alpha')=0$, hence (since $V$ is simply connected) there exists $f\in C^\infty(V)$ such that
$\alpha=\psi^*\alpha'+df$. Define
\[
  F:V\times(-\varepsilon,\varepsilon)\to \psi(V)\times(-\varepsilon,\varepsilon),\qquad F(b,t)=(\psi(b),t+f(b)).
\]
Then
\[
  F^*(dt'+\alpha')=dt+df+\psi^*\alpha'=dt+\alpha,
\qquad\text{and}\qquad
  F^*({\varphi'}^{\,2})=\varphi^2.
\]
Moreover, the metric on $U$ admits the local Kaluza--Klein form
\[
  g=\pi^*\bar g+\varphi^2\,\theta^2
  \quad\Longrightarrow\quad
  \iota^*g=\pr_1^*\bar g+\varphi^2\,(dt+\alpha)^2,
\]
and analogously $\iota'^*g'=\pr_1^*\bar g'+{\varphi'}^{\,2}(dt'+\alpha')^2$. Using $\psi^*\bar g'=\bar g$ and the identities above, we obtain
$F^*(\iota'^*g')=\iota^*g$. Therefore the map
\[
  \Phi:=\iota'\circ F\circ \iota^{-1}:\pi^{-1}(V)\to (\pi')^{-1}(\psi(V))
\]
is a local isometry and satisfies $\pi'\circ\Phi=\psi\circ\pi$, i.e.\ it is a fiber-preserving isometry. This proves the theorem.
\end{proof}

\begin{remark}[Generic finite-order separation]\label{rem:killing-generic}
In the Killing-flow class, the equivalence problem reduces to the base data $(\bar g,\varphi,\Omega)$ modulo base isometries and gauge.
At a generic point where a signature map built from a chosen list of scalar invariants has full rank, equivalence is decided by matching
finitely many scalar invariants up to a fixed order (typically $r\le 3$ in practice). Section~\ref{sec:algorithm} records a compact
recommended invariant profile and a stop criterion.
\end{remark}

\section{An algorithmic invariant test}\label{sec:algorithm}

\subsection{Recommended invariant profile (orders $0$--$3$)}\label{subsec:profile}
In practical equivalence checks one needs a compact, low-order ``signature'' before attempting higher-order prolongations.
A useful starting profile consists of the following scalar invariants (when defined):
\begin{center}
\begin{tabular}{ll}
\hline
Order $0$ & $|A|^2,\;|T|^2,\;|H|^2$\\
Order $1$ & $|\nabla H|^2,\;\div(H),\;|\nabla A|^2,\;|\nabla T|^2$\\
Order $2$ & $\Contr(R^M),\;\Contr(\nabla R^M),\;\Contr(\nabla^2 A),\;\Contr(\nabla^2 T)$\\
Order $3$ & selected contractions of $\nabla^2 R^M$ and mixed contractions with $\nabla A,\nabla T$\\
\hline
\end{tabular}
\end{center}
In a concrete model class, additional structure may yield simpler signatures. For example, in the Killing-flow orbit class
(Section~\ref{sec:killing}) one may replace part of the above list by invariants built from $(\bar g,\varphi,\Omega)$, such as
$|\nabla\ln\varphi|^2$, $\Delta\ln\varphi$, $|\Omega|^2$, $|\nabla\Omega|^2$, and base curvature scalars.

\subsection{Computational considerations (CAS-oriented)}\label{subsec:cas}
The number of components of $\nabla^k R^M$ grows rapidly with dimension, so direct high-order calculations may become expensive.
In practice one proceeds incrementally: compute a low-order signature, check rank (genericity), and increase order only if needed.
When using a computer algebra system (Maple/Mathematica/SageMath), contractions can be generated systematically by index pairing;
relations (syzygies) between invariants should be expected and can be exploited to remove redundancies. Symmetry assumptions
(e.g.\ homogeneous models) can drastically reduce the search space.

\subsection{Limitations and nongeneric cases}\label{subsec:limitations}
In high-symmetry situations the signature map may lose rank (many invariants become constant), and low-order scalar profiles may fail
to separate equivalence classes. Product submersions and the Hopf fibration in Section~\ref{sec:examples} illustrate such degeneracies.
In these cases one typically refines the analysis by stratifying jets, adding normalizations, or treating the degenerate class separately
within the moving coframes regularization framework \cite{FelsOlverII}.

For computations, one typically starts with low-order scalar invariants such as
\[
|A|^2,\quad |T|^2,\quad |H|^2,
\]
and then adds contractions of $\nabla A$, $\nabla T$ and curvature invariants from $R^M$ and $\nabla^kR^M$
as needed (Theorem~\ref{thm:gen}).

\medskip
\noindent\textbf{Algorithm (Invariant profile and equivalence check).}
\begin{enumerate}
\item Choose a local adapted orthonormal frame $(X_1,\dots,X_n;\,E_1,\dots,E_m)$.
\item Compute components $A_{ij}^{\alpha}=\langle A_{X_i}X_j,E_\alpha\rangle$ and
      $T_{\alpha\beta}^{i}=\langle T_{E_\alpha}E_\beta,X_i\rangle$.
\item Form scalar contractions (e.g.\ $|A|^2,|T|^2,|H|^2$), yielding a low-order invariant profile.
\item Compute $\nabla A$, $\nabla T$ (and if necessary $\nabla^2A,\nabla^2T$) and form additional contractions.
      Include curvature invariants from $R^M$ and $\nabla^kR^M$.
\item In the generic case, two submersions are locally equivalent if their generating invariants match up to a finite
      order $r$ (Theorem~\ref{thm:finite}).
\end{enumerate}

\section{Examples}\label{sec:examples}

\subsection{Product submersions}
Let $M=B\times F$ with product metric $g=\bar g\oplus g_F$ and projection $\pi(b,f)=b$.
Then the Levi--Civita connection splits, hence horizontal and vertical distributions are integrable and totally geodesic.
Consequently,
\[
A\equiv 0,\qquad T\equiv 0,\qquad H\equiv 0,
\]
and therefore $|A|^2=|T|^2=|H|^2=0$.

\subsection{Warped products}\label{subsec:warped}
Let $M=B\times F$ with warped product metric
\[
g=\bar g\oplus f(b)^2 g_F,\qquad f:B\to(0,\infty),
\]
and projection $\pi(b,f)=b$.
Then $\pi$ is a Riemannian submersion and the horizontal distribution is integrable (hence $A\equiv 0$).

\begin{proposition}[Warped-product identification via $T$ and $H$]\label{prop:warped-ident}
For the warped product submersion $\pi:(B\times F,\bar g\oplus f^2 g_F)\to(B,\bar g)$ one has
\[
T_UV=-\,g(U,V)\,\nabla^{\Hdist}(\ln f),\qquad
H=-\,m\,\nabla^{\Hdist}(\ln f),
\]
for all vertical $U,V$, where $m=\dim F$. Consequently,
\[
|T|^2 = m\,\big|\nabla^{\Hdist}(\ln f)\big|^2,\qquad
|H|^2 = m^2\,\big|\nabla^{\Hdist}(\ln f)\big|^2.
\]
In particular, the scalar invariants $|T|^2$ and $|H|^2$ determine $\big|\nabla^{\Hdist}(\ln f)\big|$.
Moreover, since $H$ is a canonical horizontal vector field, $\nabla^{\Hdist}(\ln f)=-(1/m)H$, hence $f$ is determined
locally up to a multiplicative constant by integrating $\nabla^{\Hdist}(\ln f)$ along the base.
\end{proposition}

\begin{proof}
The Levi--Civita connection of a warped product satisfies the standard formula
\cite{BishopONeill1969,Lee}:
for vertical vector fields $U,V$,
\[
\nabla_UV=\nabla^F_UV - g(U,V)\,\nabla(\ln f),
\]
where $\nabla(\ln f)$ is horizontal since $f$ depends only on $b\in B$. Taking horizontal projection gives
\[
T_UV=\hor(\nabla_UV)=-\,g(U,V)\,\nabla^{\Hdist}(\ln f).
\]
Let $\{E_\alpha\}_{\alpha=1}^m$ be a local orthonormal vertical frame. Then
\[
H=\sum_{\alpha=1}^m T_{E_\alpha}E_\alpha
= -\sum_{\alpha=1}^m g(E_\alpha,E_\alpha)\,\nabla^{\Hdist}(\ln f)
= -m\,\nabla^{\Hdist}(\ln f).
\]
The norm identities follow by contraction. The final claim follows from $\nabla^{\Hdist}(\ln f)=-(1/m)H$ and integration.
\end{proof}
\begin{theorem}[Reconstruction of the warping function]\label{thm:warped-reconstruct}
In the warped-product submersion of Proposition~\ref{prop:warped-ident}, the warping function $f$ is determined locally up to a
multiplicative constant by the invariant mean curvature field $H$. More precisely, writing $u=\ln f$ one has
\[
  \nabla^{\Hdist}u = -\frac1m H,
\]
and hence $u$ (and therefore $f$) is obtained locally by integrating $H$ along the base.
\end{theorem}

\begin{proof}
This is immediate from $H=-m\nabla^{\Hdist}(\ln f)$ in Proposition~\ref{prop:warped-ident}.
\end{proof}

\begin{theorem}[Equivalence of warped-product submersions]\label{thm:warped-equivalence}
Let $\pi:(B\times F,\bar g\oplus f^2 g_F)\to(B,\bar g)$ and $\pi':(B'\times F',\bar g'\oplus {f'}^{\,2} g_{F'})\to(B',\bar g')$
be warped-product submersions. Then $\pi$ and $\pi'$ are locally equivalent under fiber-preserving isometries if and only if there exist
a local isometry $\psi:(B,\bar g)\to(B',\bar g')$, a diffeomorphism $\phi:F\to F'$, and a constant $c>0$ such that
$\phi:(F,g_F)\to(F',c^2 g_{F'})$ is an isometry and
\[
  f' \circ \psi = c\, f.
\]
In particular, if $(F,g_F)$ and $(F',g_{F'})$ are (locally) isometric, one can take $c=1$ and the condition reduces to $f'\circ\psi=f$.
\end{theorem}

\begin{proof}
(\emph{Necessity.}) Suppose $\pi$ and $\pi'$ are locally equivalent under fiber-preserving isometries. Then there exist local isometries
$\Phi:(B\times F,\bar g\oplus f^2 g_F)\to(B'\times F',\bar g'\oplus {f'}^{\,2} g_{F'})$ and $\psi:(B,\bar g)\to(B',\bar g')$ such that
$\pi'\circ\Phi=\psi\circ\pi$. For each $b\in B$, the restriction of $\Phi$ to the fiber $\{b\}\times F$ is an isometry between the induced
fiber metrics, hence
\[
  (F,\,f(b)^2 g_F)\ \cong\ (F',\,f'(\psi(b))^2 g_{F'}).
\]
In particular, the mean curvature field $H$ is preserved by Lemma~\ref{lem:naturality-ATH}. Using Proposition~\ref{prop:warped-ident} we get
$H=-m\nabla^{\Hdist}(\ln f)$ and $H'=-m\nabla^{\Hdist}(\ln f')$, so
\[
  \psi^*d(\ln f') = d(\ln f).
\]
Therefore $\ln f - \psi^*(\ln f')$ is constant on a connected neighborhood, i.e.\ there exists $c>0$ such that
$f'\circ\psi = c\,f$ locally.

Fix a point $b_0\in B$. Restricting $\Phi$ to the fiber over $b_0$ yields an isometry
$(F,f(b_0)^2 g_F)\cong(F',f'(\psi(b_0))^2 g_{F'})$. Dividing both metrics by $f(b_0)^2$ and using $f'(\psi(b_0))=c f(b_0)$ gives an isometry
\[
  (F,g_F)\ \cong\ (F',c^2 g_{F'}).
\]
Let $\phi:F\to F'$ denote such a (local) isometry. This proves necessity.

(\emph{Sufficiency.}) Conversely, assume there exist $\psi$, $\phi$ and $c>0$ as in the statement. Define
\[
  \Phi:B\times F\to B'\times F',\qquad \Phi(b,p)=(\psi(b),\phi(p)).
\]
Then $\pi'\circ\Phi=\psi\circ\pi$ and
\[
  \Phi^*(\bar g'\oplus {f'}^{\,2} g_{F'})
  =\psi^*\bar g' \oplus (f'\circ\psi)^2\,\phi^* g_{F'}
  =\bar g \oplus (c f)^2 \cdot \frac{1}{c^2}\, g_F
  =\bar g\oplus f^2 g_F,
\]
where we used $\psi^*\bar g'=\bar g$ and $\phi^*(c^2 g_{F'})=g_F$. Hence $\Phi$ is a fiber-preserving isometry, and the submersions are equivalent.
\end{proof}

\subsection{The Hopf fibration (benchmark)}
The Hopf fibration
\[
\pi:\Sphere^{2n+1}\to \CP^{n}
\]
(with standard metrics) is a classical Riemannian submersion; fibers are $S^1$-orbits \cite{Besse,BairdWood}.
Fibers are great circles in $\Sphere^{2n+1}$, hence totally geodesic.
By Lemma~\ref{lem:T0-totgeod} and Lemma~\ref{lem:H0-minimal},
\[
T\equiv 0,\qquad H\equiv 0.
\]
On the other hand, the horizontal distribution is not integrable (it is related to the standard contact/Sasakian structure),
so $A\not\equiv 0$ (Lemma~\ref{lem:A0-integrable}).
By homogeneity, the scalar invariant $|A|^2$ is constant; see \cite{Besse,BairdWood} for details. See also \cite{AlekseevskyNikonorov2009,AlekseevskyArvanitoyeorgos2007} for homogeneous-geodesic perspectives on related homogeneous models.

\section{Conclusion}
\label{sec:conclusion}
We formulated the local equivalence problem for Riemannian submersions under fiber-preserving isometries and reviewed the standard
O'Neill formalism ($A$, $T$, $H$) together with curvature identities needed for invariant constructions.
The general moving-frames finiteness and generation statements were included as background.

The main new results are obtained in concrete submersion classes.
First, for orbit submersions induced by a nowhere-vanishing Killing flow (Section~\ref{sec:killing}) we derived explicit formulas for
$A$ and $H$ in terms of the Killing length $\varphi=|K|$ and the connection curvature $\Omega$, and proved an equivalence criterion
under fiber-preserving isometries purely in terms of the base data $(\bar g,\varphi,\Omega)$.
Second, for warped-product submersions we stated and proved reconstruction and equivalence results showing how the warping function is
recovered (up to scale) from the invariant mean-curvature field and when two warped-product submersions are equivalent.
Finally, we provided a compact recommended low-order invariant profile, CAS-oriented computational guidance, and a discussion of
nongeneric/high-symmetry situations (e.g.\ product and Hopf cases) in which low-order scalar signatures may fail to separate equivalence
classes and must be refined.

These results supply both a conceptual bridge between submersion geometry and differential-invariant methods and a practical toolkit
for classification and parameter identification in geometric models.


\begin{thebibliography}{99}

\bibitem{ONeill1966}
B.~O'Neill,
The fundamental equations of a submersion,
\emph{Michigan Math. J.} 13 (1966) 459--469.

\bibitem{ONeillBook}
B.~O'Neill,
\emph{Semi-Riemannian Geometry with Applications to Relativity},
Academic Press, 1983.

\bibitem{Besse}
A.~L.~Besse,
\emph{Einstein Manifolds},
Springer, 1987.
https://doi.org/10.1007/978-3-540-74311-8.

\bibitem{BishopONeill1969}
R.~L.~Bishop, B.~O'Neill,
Manifolds of negative curvature,
\emph{Trans. Amer. Math. Soc.} 145 (1969) 1--49.
https://doi.org/10.2307/1995057.

\bibitem{BairdWood}
P.~Baird, J.~C.~Wood,
\emph{Harmonic Morphisms Between Riemannian Manifolds},
Oxford University Press, 2003.

\bibitem{doCarmo}
M.~P.~do~Carmo,
\emph{Riemannian Geometry},
Birkh\"auser, 1992.

\bibitem{Lee}
J.~M.~Lee,
\emph{Riemannian Manifolds: An Introduction to Curvature},
Springer, 1997.

\bibitem{Pet}
P.~Petersen,
\emph{Riemannian Geometry}, 3rd ed.,
Springer (GTM 171), 2016.
https://doi.org/10.1007/978-3-319-26654-1.

\bibitem{KobNom}
S.~Kobayashi, K.~Nomizu,
\emph{Foundations of Differential Geometry}, Vols.~I--II,
Wiley, 1963/1969.

\bibitem{Cartan}
\'E.~Cartan,
Les probl\`emes d'\'equivalence,
in: \emph{Oeuvres compl\`etes}, Partie~II, Vol.~2,
Gauthier-Villars, Paris, 1953, pp.~1311--1334.

\bibitem{OlverBook}
P.~J.~Olver,
\emph{Equivalence, Invariants and Symmetry},
Cambridge University Press, 1995.
https://doi.org/10.1017/CBO9780511609565.

\bibitem{FelsOlverI}
M.~Fels, P.~J.~Olver,
Moving coframes: I. A practical algorithm,
\emph{Acta Appl. Math.} 51 (1998) 161--213.
https://doi.org/10.1023/A:1005878210297.

\bibitem{FelsOlverII}
M.~Fels, P.~J.~Olver,
Moving coframes: II. Regularization and theoretical foundations,
\emph{Acta Appl. Math.} 55(2) (1999) 127--208.
https://doi.org/10.1023/A:1006195823000.

\bibitem{AlekseevskyGeomI}
D.~V.~Alekseevsky, A.~M.~Vinogradov, V.~V.~Lychagin,
\emph{Geometry I: Basic Ideas and Concepts of Differential Geometry},
Encyclopaedia of Mathematical Sciences, Vol.~28,
Springer-Verlag, 1991.



\bibitem{AlekseevskyMichor1995}
D.V.~Alekseevsky, P.W.~Michor,
\newblock Differential geometry of Cartan connections,
\newblock \emph{Publicationes Mathematicae Debrecen} 47 (1995) 349--375.

\bibitem{AlekseevskyNikonorov2009}
D.V.~Alekseevsky, Yu.G.~Nikonorov,
\newblock Compact Riemannian manifolds with homogeneous geodesics,
\newblock \emph{SIGMA} 5 (2009) 093.

\bibitem{AlekseevskyArvanitoyeorgos2007}
D.V.~Alekseevsky, A.~Arvanitoyeorgos,
\newblock Riemannian flag manifolds with homogeneous geodesics,
\newblock \emph{Trans. Amer. Math. Soc.} 359 (2007) 3769--3789.

\bibitem{AlekseevskyEtAl2022}
D.V.~Alekseevsky, J.~Gutt, G.~Manno, G.~Moreno,
\newblock A general method to construct invariant PDEs on homogeneous manifolds,
\newblock \emph{Commun. Contemp. Math.} 24 (2022) 2050089.

\bibitem{SharipovSalimov2025}
X.~Sharipov, E.~Salimov,
Differential invariants of submersions with respect to transformation groups,
\emph{J. Geom. Symmetry Phys.} 74 (2025) 101--111.
https://doi.org/10.7546/jgsp-74-2025-101-111.

\bibitem{NarmanovSharipov2021}
A.~Narmanov, X.~Sharipov,
On the geometry of submersions,
\emph{Geometry, Integrability and Quantization} 22 (2021) 199--208.
https://doi.org/10.7546/giq-22-2021-199-208.

\bibitem{SharipovAliboyevKhalimov2024}
X.~Sharipov, S.~Aliboyev, U.~Khalimov,
Geometry of Riemannian submersions and differential invariants,
\emph{AIP Conf. Proc.} 3147(1) (2024) 020003.
https://doi.org/10.1063/5.0210288.

\bibitem{Guzal2020}
G.~Abdishukurova, A.~Narmanov, X.~Sharipov,
Differential invariants of one parametrical group of transformations,
\emph{Mathematics and Statistics} 8(3) (2020) 347--352.
https://doi.org/10.13189/ms.2020.080314.

\bibitem{Almuratov2024}
F.~Almuratov, X.~Sharipov, S.~Alimov,
The coupling constant threshold effects for one-particle Schr\"odinger-type operators in lattice,
\emph{J. Geom. Symmetry Phys.} 68 (2024) 1--19.
https://doi.org/10.7546/jgsp-68-2024-1-19.

\bibitem{RajabovSharipov2021}
E.~O.~Rajabov, Kh.~F.~Sharipov,
On the geometry of orbit of conformal vector fields,
\emph{Bulletin of the Institute of Mathematics} 4(1) (2021) 38--45.

\bibitem{Fuglede1978}
B.~Fuglede,
Harmonic morphisms between Riemannian manifolds,
\emph{Ann. Inst. Fourier} 28(2) (1978) 107--144.
https://doi.org/10.5802/aif.691.

\bibitem{FalcitelliPastoreIanus2004}
M.~Falcitelli, S.~Ianu\c{s}, A.~M.~Pastore,
\emph{Riemannian Submersions and Related Topics},
World Scientific, 2004.
https://doi.org/10.1142/5568.

\bibitem{Sahin2017}
B.~Sahin,
\emph{Riemannian Submersions, Riemannian Maps in Hermitian Geometry, and their Applications},
Academic Press, 2017.

\bibitem{Ou2024}
Y.-L.~Ou,
A short survey on biharmonic Riemannian submersions,
\emph{International Electronic Journal of Geometry} 17(1) (2024) 259--266.
https://doi.org/10.36890/iejg.1429642.

\end{thebibliography}
\end{document}